\documentclass[12pt]{amsart}
\usepackage{amssymb,amsmath,amsfonts, amsthm}
\usepackage{graphicx}
\usepackage{tikz}

\newcommand{\R}{\mathbb{R}}
\newcommand{\N}{\mathbb{N}}

\newcommand{\Q}{\mathbb{Q}}
\newcommand{\C}{\mathbb{C}}
\newcommand{\bbH}{\mathbb{H}}

\newcommand{\Diag}{\mathrm{Diag}}
\newcommand{\tr}{\mathrm{tr}}

\newcommand{\Tr}{\mathrm{Tr}}

\newcommand{\SL}{\mathrm{SL}}
\newcommand{\PSL}{\mathrm{PSL}}

\newcommand{\Ha}{\mathrm{\textbf{H}}}
\newcommand{\RP}{\mathbb{RP}}
\newcommand{\Lim}{\mathcal{L}}

\newtheorem{theorem}{Theorem}[section]
\newtheorem{proposition}[theorem]{Proposition}
\newtheorem{corollary}[theorem]{Corollary}
\newtheorem{lemma}[theorem]{Lemma}

\newtheorem*{thmA}{Theorem A}
\newtheorem*{thmB}{Theorem B}

\theoremstyle{definition}

\theoremstyle{remark}
\newtheorem*{rem}{Remark}

\newcommand{\geoBig}{
(0,-1) arc(0:60:1.73)
}
\newcommand{\geoMedium}{
(0,-1) arc(0:90:1)
}
\newcommand{\geoSmall}{
(0,-1) arc(0:120:0.58)
}
\newcommand{\geoTiny}{
(0,-1) arc(0:150:0.27)
}
\newcommand{\geoVeryTiny}{
(0,-1) arc(0:165:0.13165)
}
\newcommand{\geoLine}{
(0,-1) --(0,1)
}

\title{The limit sets of subgroups of lattices in $\PSL(2,\R)^r$}
\author{Slavyana Geninska}

\address{Institut de Math\'ematiques de Toulouse \\ 
Universit\'e Toulouse 3 \\
118 route de Narbonne \\
31062 Toulouse, France
}
\email{geninska@math.univ-toulouse.fr}

\begin{document}
\begin{abstract}
In this paper we give a description of the possible limit sets of finitely generated subgroups of irreducible lattices in $\PSL(2,\R)^r$.
\end{abstract}

\maketitle

\pagestyle{plain}
\setcounter{page}{1}

\section{Introduction}

While lattices in higher rank Lie groups are fairly well understood, only little is known about discrete subgroups of infinite covolume of semi-simple Lie groups. The main class of examples are Schottky groups. 

This paper continues the study of the limit set of finitely generated subgroups of irreducible lattices in $\PSL(2,\R)^r$ that was started in the author's papers \cite{sG10a} and \cite{sG12}. In the first paper we give necessary and sufficient conditions for the limit set to be ``small'' and in the second paper we give examples of infinite covolume subgroups that have ``big'' limit sets.

In this paper we give a description of the possible limit sets of finitely generated subgroups of irreducible lattices in $\PSL(2,\R)^r$. We see that morally we do not have limit sets of surprising types and that the examples of \cite{sG10a} and \cite{sG12} are in some sense extremal.

This article is organized as follows. The second section fixes the notation and gives the necessary background. We provide a detailed description of the geometric boundary of $(\bbH^2)^r$, which is the set of equivalence classes of asymptotic geodesic rays. We introduce the notion of the limit set as the part of the orbit closure $\overline{\Gamma(x)}$ in the geometric boundary where $x$ is an arbitrary point in $(\bbH^2)^r$. We also state a natural structure theorem for the regular limit set $\Lim_\Gamma^{reg}$ of discrete nonelementary groups $\Gamma$ due to Link~\cite{gL02}: $\Lim_\Gamma^{reg}$ is the product of the Furstenberg limit set $F_\Gamma$ and the projective limit set $P_\Gamma$.

In Section 3 we study the limit set of finitely generated Zariski dense subgroups. The main result is the following theorem (Theorem~\ref{T:FurstenbergLimitSets} in the text). 

\begin{thmA}
Let $\Gamma$ be a finitely generated Zariski dense subgroup of an irreducible arithmetic subgroup of $\PSL(2,\R)^r$ with $r \geq 2$.

If there is $j \in \{1,\ldots,r\}$ such that $p_j(\Gamma)$ is discrete, then $F_\Gamma$ is homeomorphic to $\Lim_{p_j(\Gamma)}\times (\partial \bbH^2)^m$, where $m$ is the number of nondiscrete projections of $\Gamma$.

Otherwise, if $p_i(\Gamma)$ is not discrete for all $i=1,\ldots,r$, then $F_\Gamma$ is  the whole Furstenberg boundary.
\end{thmA}

This theorem together with the fact that the projective limit set of a Zariski dense subgroup of $\PSL(2,\R)^r$ is of nonempty interior (this is a special case of the theorem by Benoist in Section 1.2 in \cite{yB97}) describes the regular limit set of finitely generated Zariski dense subgroups of $\PSL(2,\R)^r$.

In Section 4 we treat the general case of finitely generated subgroups. The main result is the following theorem which is a compilation of Proposition~\ref{P:FurstenbergLimitSets} and Proposition~\ref{P:ProjectiveLimitSets} in the text.

\begin{thmB}
Let $\Gamma$ be a finitely generated and nonelementary subgroup of an irreducible arithmetic subgroup $\Delta$ of $\PSL(2,\R)^r$ with $r \geq 2$. Further let $m$ be the number of nondiscrete projections of $\Gamma$ and $k$ the degree of extension $k = [\Q(\Tr(p_j(\Delta^{(2)})) : \Q(\Tr(p_j(\Gamma^{(2)})))]$
for one and hence for all $j \in \{1,\ldots,r\}$.
\begin{itemize}
\item[(i)] If there is $j \in \{1,\ldots,r\}$ such that $p_j(\Gamma)$ is discrete, then $F_\Gamma$ is homeomorphic to $\Lim_{p_j(\Gamma)}\times (\partial \bbH^2)^\frac{m}{k}$, otherwise,  $F_\Gamma$ is homeomorphic to $(\partial \bbH^2)^\frac{m}{k}$.
\item[(ii)] The projective limit set $P_{\Gamma}$ is of dimension $\frac{r}{k}-1$.
\end{itemize}
\end{thmB}

A corollary of Theorem B is the fact that $\dim{F_{\Gamma}} \leq 1+\dim{P_\Gamma}$, i.e. it is not possible to have simultaneously a big $F_\Gamma$ and a small $P_{\Gamma}$.

At the end of the section we give some examples of subgroups of irreducible arithmetic groups in $\PSL(2,\R)^r$ and their limit sets.

The author thanks Fanny Kassel and Francois Gueritaud for a motivating discussion and the anonymous referee for valuable comments and suggestions.

\section{Background}
In this section we provide some basic facts and notations that are needed in the rest of this paper. 

We will change freely between matrices in $\SL(2,\R)$ and their action as fractional linear transformations, namely as elements in $\PSL(2,\R)$.

For $g = \begin{bmatrix} a&b\\c&d \end{bmatrix} \in \PSL(2,\R)$ we set $\tr(g)=|a+d|$. 

For a subgroup $\Gamma$ of $\PSL(2,\R)$ we call 
\[ \Tr(\Gamma) = \{\tr(g) \mid g \in \Gamma \} \]
the \textit{trace set of} $\Gamma$. 

The \textit{translation length} $\ell(g)$ of a hyperbolic $g$ is the distance between a point $x$ on the geodesic fixed by $g$ and its image $g(x)$ under $g$. If $g$ is elliptic, parabolic or the identity, we define $\ell(g):=0$.

The group $\Gamma$ is \textit{elementary} if it has a finite orbit in its action on $\bbH^2\cup\R\cup\{\infty\}$. Otherwise it is said to be \textit{nonelementary}. Every nonelementary subgroup of $\PSL(2,\R)$ contains infinitely many hyperbolic elements, no two of which have a common fixed point (see Theorem 5.1.3 in the book of Beardon~\cite{aB95}).

A Schottky group is a finitely generated free subgroup of $\PSL(2,\R)$ that contains only hyperbolic isometries except for the identity. We will mainly deal with two-generated Schottky groups.

For each two hyperbolic isometries without common fixed points, we can find powers of them that generate a Schottky group. This means that every nonelementary subgroup of $\PSL(2,\R)$ has a subgroup that is a Schottky group. A proof of this lemma can be found in \cite{sG09}.

A Schottky group contains isometries without common fixed points because it is nonelementary.

\subsection{The geometric boundary of $(\bbH^2)^r$}
For $i=1,\ldots,r$, we denote by $p_i: (\bbH^2)^r \rightarrow \bbH^2$, $(z_1,...,z_{r}) \mapsto z_i$ the $i$-th projection of $(\bbH^2)^r$ into $\bbH^2$. Let $\gamma: [0, \infty )\rightarrow (\bbH^2)^r$ be a curve in $(\bbH^2)^r$. Then $\gamma$ is a geodesic ray in $(\bbH^2)^r$ if and only if $p_i\circ \gamma$ is a geodesic ray or a point in $\bbH^2$ for each $i = 1, \ldots , r$. A geodesic $\gamma$ is \textit{regular} if $p_i\circ \gamma$ is a nonconstant geodesic in $\bbH^2$ for each $i = 1, \ldots , r$.

Two unit speed geodesic rays $\gamma$ and $\delta$ in $(\bbH^2)^r$ are said to be asymptotic if there exists a positive number $c$ such that $d(\gamma(t), \delta(t)) \leq c$ for all $t \geq 0$. This is an equivalence relation on the unit speed geodesic rays of $(\bbH^2)^r$. For any unit speed geodesic $\gamma$ of $(\bbH^2)^r$ we denote by $\gamma(+\infty)$ the equivalence class of its positive ray. 

We denote by $\partial ((\bbH^2)^r)$ the set of all equivalence classes of unit speed geodesic rays of $(\bbH^2)^r$. We call $\partial ((\bbH^2)^r)$ the \textit{geometric boundary} of $(\bbH^2)^r$. 

The geometric boundary $\partial ((\bbH^2)^r)$ with the cone topology is homeomorphic to the unit tangent sphere of a point in $(\bbH^2)^r$ (see Eberlein \cite{pE96}, 1.7). (For example $\partial \bbH^2$ is homeomorphic to $S^1$.) The homeomorphism is given by the fact that for each point $x_0$ and each unit speed geodesic ray $\gamma$ in $(\bbH^2)^r$ there exists a unique unit speed geodesic ray $\delta$ with $\delta(0) = x_0$ which is asymptotic to $\gamma$.

The group $\PSL(2,\R)^r$ acts on $(\bbH^2)^r$ by isometries in the following way. For $g = (g_1,\ldots, g_{r}) \in \PSL(2,\R)^r$
$$ g:(\bbH^2)^r \rightarrow (\bbH^2)^r, \quad (z_1, \ldots, z_{r}) \mapsto (g_1 z_1, \ldots, g_{r} z_{r}), $$
where $z_i \mapsto g_i z_i$ is the usual action given by linear fractional transformation, $i = 1, \ldots, {r}$.

The action of $\PSL(2,\R)^r$ can be extended naturally to $\partial ((\bbH^2)^r)$. Let $g$ be in $\PSL(2,\R)^r$ and $\xi$ be a point in the boundary $\partial ((\bbH^2)^r)$. If $\gamma$ is a representative of $\xi$, then $g(\xi)$ is the equivalence class of the geodesic ray $g\circ \gamma$.

We call $g$ \textit{elliptic} if all $g_i$ are elliptic isometries, \textit{parabolic} if all $g_i$ are parabolic isometries and \textit{hyperbolic} if all $g_i$ are hyperbolic isometries. In all the other cases we call $g$ \textit{mixed}. 

If at least one $\ell(g_i)$ is different from zero, then we define the \textit{translation direction} of $g$ as $L(g):= (\ell(g_1): \ldots: \ell(g_{r})) \in \RP^{r-1}$.

\subsection{Decomposition of the geometric boundary of $(\bbH^2)^r$}
In this section we show a natural decomposition of the geometric boundary of $(\bbH^2)^r$ and in particular of its regular part. This is a special case of a general construction for a large class of symmetric spaces (see e.g. Leuzinger~\cite{eL92} and Link~\cite{gL02}). This decomposition plays a main role in this article.

Let $x=(x_1,\ldots,x_{r})$ be a point in $(\bbH^2)^r$. We consider the \textit{Weyl chambers} with vertex $x$ in $(\bbH^2)^r$ given by the product of the images of the geodesics $\delta_i:[0,\infty)\rightarrow \bbH^2$ with $\delta_i(0)=x_i$ for $i = 1,\ldots,r$. The isotropy group in $\PSL(2,\R)^r$ of $x$ is $\mathrm{PSO}(2)^r$. It acts simply transitively on the Weyl chamber with vertex $x$. 

Let $W$ be a Weyl chamber with vertex $x$. In $W$, two unit speed geodesics $\gamma(t) = (\gamma_1(t),\ldots,\gamma_{r}(t))$ and $\tilde{\gamma} = (\tilde{\gamma}_1(t),\ldots,\tilde{\gamma}_{r}(t))$ are different if and only if the corresponding projective points $$\left(d_H(\gamma_1(0),\gamma_1(1)):\ldots:d_H(\gamma_{r}(0),\gamma_{r}(1))\right) \text{ and}$$ $$(d_H(\tilde{\gamma}_1(0),\tilde{\gamma}_1(1)):\ldots:d_H(\tilde{\gamma}_{r}(0),\tilde{\gamma}_{r}(1)))$$ 
are different. Here $d_H$ denotes the hyperbolic distance in $\bbH^2$. The point in $\RP^{r-1}$ given by $\left(d_H(\gamma_1(0),\gamma_1(1)):\ldots:d_H(\gamma_{r}(0),\gamma_r(1))\right)$ is a direction in the Weyl chamber and it is the same as $(\left\|v_1\right\|:\ldots:\left\|v_{r}\right\|)$, where $v = (v_1,\ldots,v_{r}):= \gamma'(0)$ is the unit tangent vector of $\gamma$ in $0$. 

In other words we can extend the action of $Iso_x$ to the tangent space at $x$ in $(\bbH^2)^r$. Then $Iso_x$ maps a unit tangent vector at $x$ onto a unit tangent vector at $x$.
Let $v$ be a unit tangent vector at $x$ in $(\bbH^2)^r$. We denote by $v_i$ the $i$-th projection of $v$ on the tangent spaces at $x_i$, $i=1,\ldots,r$. Then all the vectors $w$ in the orbit of $v$ under $Iso_x$ have $\left\|w_i\right\|=\left\|v_i\right\|$. 

Let $v$ be a vector in the unit tangent sphere at $x$ in $(\bbH^2)^r$. If $v$ is tangent to a regular geodesic, then the orbit of $v$ is homeomorphic to $(S^1)^r \cong \left( \partial \bbH^2 \right)^r $ because $ \partial \bbH^2 \cong S^1$. The orbit of $v$ under the group $\mathrm{PSO}(2)^r$ consists of all unit tangent vectors $w$ at $x$ such that $\left\|w_i\right\|=\left\|v_i\right\|$ for $i=1,\ldots,r$.

The \textit{regular boundary} $\partial ((\bbH^2)^r)_{reg}$ of $(\bbH^2)^r$ consists of the equivalence classes of regular geodesics. Hence it is identified with $\left( \partial \bbH^2 \right)^r \times \RP^{r-1}_+$ where 
$$ \RP^{r-1}_+ := \left\{(w_1:\ldots:w_{r}) \in \RP^{r-1} \mid w_1 > 0, \ldots, w_{r} > 0 \right\}. $$
Here $w_1,..,w_{r}$ can be thought as the norms of the projections of the regular unit tangent vectors on the simple factors of $(\bbH^2)^r$.

$\left( \partial \bbH^2 \right)^r$ is called the \textit{Furstenberg boundary} of $(\bbH^2)^r$. 

We note that the decomposition of the boundary into orbits under the group $Iso_x$ is independent of the point $x$.

\subsection{The limit set of a group}
Let $x$ be a point and $\{x_n\}_{n\in \N}$ a sequence of points in $(\bbH^2)^r$. We say that $\{x_n\}_{n\in \N}$ converges to a point $\xi \in \partial\left((\bbH^2)^r\right)$ if $\{x_n\}_{n\in \N}$ is discrete in $(\bbH^2)^r$ and the sequence of geodesic rays starting at $x$ and going through $x_n$ converges towards $\xi$ in the cone topology. With this topology, $(\bbH^2)^r \cup \partial\left((\bbH^2)^r\right)$ is a compactification of $(\bbH^2)^r$.

Let $\Gamma$ be a subgroup of $\PSL(2,\R)^r$. We denote by $\Gamma(x)$ the orbit of $x$ under $\Gamma$ and by $\overline{\Gamma(x)}$ - its closure. The \textit{limit set} of $\Gamma$ is $\mathcal{L}_\Gamma:= \overline{\Gamma(x)}\cap \partial\left((\bbH^2)^r\right)$. The limit set is independent of the choice of the point $x$ in $(\bbH^2)^r$. The \textit{regular limit set} is $\mathcal{L}_\Gamma^{reg}:=\mathcal{L}_\Gamma \cap  \partial\left((\bbH^2)^r\right)_{reg}$ and the \textit{singular limit set} is $\mathcal{L}_\Gamma^{sing}:=\mathcal{L}_\Gamma \backslash \mathcal{L}_\Gamma^{reg}$. 

We denote by $F_\Gamma$ the projection of $\mathcal{L}_\Gamma^{reg}$ on the Furstenberg boundary $\left( \partial \bbH^2 \right)^r$ and by $P_\Gamma$ the projection of $\mathcal{L}_\Gamma^{reg}$ on $\RP^{r-1}_+$. The projection $F_\Gamma$ is the \textit{Furstenberg limit set} of $\Gamma$ and $P_\Gamma$ is the \textit{projective limit set} of $\Gamma$.
\medskip

Let $h\in \Gamma$ be a hyperbolic element or a mixed one with only hyperbolic or elliptic components. We consider the unit speed geodesics $\gamma$ in $(\bbH^2)^r$ such that $h\circ\gamma(t) = \gamma(t + T_h)$ for a fixed $T_h \in \R_{>0}$ and all $t \in \R$. Their union forms a flat of dimension the number of hyperbolic components of $h$. This flat is the product of the axes of the hyperbolic components and the fixed points of the elliptic components. Hence the considered geodesics are parallel in the Euclidean sense and therefore equivalent. 

We take a geodesic $\gamma$ as above. For $y \in \gamma$, the sequence $h^n(y)$ converges to $\gamma(+\infty)$. Hence also for every $x\in (\bbH^2)^r$, the sequence $h^n(x)$ converges to $\gamma(+\infty)$. Thus $\gamma(+\infty)$ is in $\mathcal{L}_\Gamma$. The sequence $h^{-n}(x)$ converges to $\gamma(-\infty):=-\gamma(+\infty)$ and therefore $\gamma(-\infty)$ is also in $\mathcal{L}_\Gamma$. The points $\gamma(+\infty)$ and $\gamma(-\infty)$ are the only fixed points of $h$ in $\Lim_\Gamma$. The point $\gamma(+\infty)$ is the \textit{attractive} fixed point of $h$ and the point $\gamma(-\infty)$ - the \textit{repulsive} fixed point of $h$. 

If $h$ is hyperbolic, then for all $i=1,\ldots,r$, the projection $p_i \circ \gamma$ is not a point. Hence $\gamma$ is regular and $\gamma(+\infty) \in \mathcal{L}_\Gamma^{reg}$. The point $\gamma(+\infty)$ can be written as $(\xi_F,\xi_P)$ in our description of the regular geometric boundary where 
$$\xi_F := (p_1\circ \gamma(+\infty),\ldots,p_{r}\circ \gamma(+\infty))$$ 
is in the Furstenberg boundary and 
$$\xi_P := (d_H(p_1\circ \gamma(0),p_1\circ \gamma(1)) : \ldots :d_H(p_{r}\circ \gamma(0), p_{r}\circ \gamma(1)))$$ 
is in the projective limit set. Here we note that $\xi_P$ is also equal to 
$$(d_H(p_1\circ \gamma(0),p_1\circ \gamma(T_h)) : \ldots :d_H(p_{r}\circ \gamma(0), p_{r}\circ \gamma(T_h))),$$ 
which is exactly the translation direction of $h$.

Thus the translation direction of each hyperbolic isometry $h$ in $\Gamma$ determines a point in the projective limit set $P_\Gamma$. This point does not change after conjugation with $h$ or after taking a power $h^m$ of $h$, because in these cases the translation direction remains unchanged. 

\medskip

Recall that following Maclachlan and Reid \cite{cM03}, we call a subgroup $\Gamma$ of $\PSL(2,\R)$ \textit{elementary} if there exists a finite $\Gamma$-orbit in $\overline{\bbH^2}:=\bbH^2 \cup \partial \bbH^2$ and \textit{nonelementary} if it is not elementary. Since $\bbH^2$ and $\partial \bbH^2$ are $\Gamma$-invariant, any $\Gamma$-orbit of a point in $\overline{\bbH^2}$ is either completely in $\bbH^2$ or completely in $\partial \bbH^2$. 

We call a subgroup $\Gamma$ of $\PSL(2,\R)^r$ \textit{nonelementary} if for all $i = 1, \ldots, r$,  $p_i(\Gamma)$ is nonelementary, and if for all $g \in \Gamma$ that are mixed, the projections $p_i\circ g$ are either hyperbolic or elliptic of infinite order. Since for all $i = 1, \ldots, r$,  $p_i(\Gamma)$ is nonelementary, $\Gamma$ does not contain only elliptic isometries and thus $\Lim_\Gamma$ is not empty.

This definition of nonelementary is more restrictive than the one given by Link in \cite{gL02}. By Lemma~1.2 in \cite{sG10a} if a subgroup $\Gamma$ of $\PSL(2,\R)^r$ is nonelementary (according to our definition), then it is nonelementary in the sense of Link's definition in \cite{gL02}.

The next theorem is a special case of Theorem 3 from the introduction of \cite{gL02}. It describes the structure of the regular limit set of nonelementary discrete subgroups of $\PSL(2,\R)^r$.

\begin{theorem}[\cite{gL02}]
\label{T:LinkFP}
Let $\Gamma$ be a nonelementary discrete subgroup of the group $\PSL(2,\R)^r$ acting on $(\bbH^2)^r$. If $\mathcal{L}_\Gamma^{reg}$ is not empty, then $F_\Gamma$ is a minimal closed $\Gamma$-invariant subset of $(\partial\bbH^2)^r$, the regular limit set equals the product $F_\Gamma \times P_\Gamma$ and $P_\Gamma$ is equal to the closure in $\RP^{r-1}_+$ of the set of translation directions of the hyperbolic isometries in $\Gamma$. 
\end{theorem}

\subsection{Irreducible arithmetic groups in $\PSL(2,\R)^r$}
In this section, following Schmutz and Wolfart \cite{pS00} and Borel \cite{aB81}, we will describe the irreducible arithmetic subgroups of $\PSL(2,\R)^r$.

Let $K$ be a totally real algebraic number field of degree $n = [K:\Q]$ and let $\phi_i$, $i \in \{1, \ldots, n\}$, be the $n$ distinct embeddings of $K$ into $\R$, where $\phi_1 = id$.

Let $A = \left(\frac{a,b}{K}\right)$ be a quaternion algebra over $K$ such that for $1 \leq i \leq r$, the quaternion algebra $\left(\frac{\phi_i(a),\phi_i(b)}{\R}\right)$ is \textit{unramified}, i.e. isomorphic to the matrix algebra $M(2, \R)$, and for $r < i \leq n$, it is \textit{ramified}, i.e. isomorphic to the Hamilton quaternion algebra $\Ha$. In other words, the embeddings 
\[\phi_i: K \longrightarrow \R, \quad i = 1,\ldots, r\] 
can be extended to embeddings of $A$ into $M(2,\R)$ and the embeddings 
\[\phi_i: K \longrightarrow \R, \quad i = r+1,\ldots, n\] 
can be extended to embeddings of $A$ into $\Ha$. Note that for $\phi_i$, $i = 1,\ldots, r$, the identifications of $\left(\frac{\phi_i(a),\phi_i(b)}{\R}\right)$ with the matrix algebra  $M(2,\R)$ are not canonical. 

Let $\mathcal{O}$ be an order in $A$ and $\mathcal{O}^1$ the group of units in $\mathcal{O}$. Define $\Gamma(A,\mathcal{O}):= \phi_1(\mathcal{O}^1) \subset \SL(2,\R)$. The canonical image of $\Gamma(A,\mathcal{O})$ in $\PSL(2,\R)$ is called a group \textit{derived from a quaternion algebra}. The group $\Gamma(A,\mathcal{O})$ acts by isometries on $(\bbH^2)^r$ as follows. An element $g = \phi_1(\varepsilon)$ of $\Gamma(A,\mathcal{O})$ acts via
\[g: (z_1, \ldots, z_{r}) \mapsto (\phi_1(\varepsilon)z_1, \ldots, \phi_{r}(\varepsilon)z_{r}), \]
where $z_i \mapsto \phi_i(\varepsilon)z_i$ is the usual action by linear fractional transformation, $i = 1, \ldots, {r}$.

For a subgroup $S$ of $\Gamma(A,\mathcal{O})$ we denote by $S^*$ the group
$$\{g^*:= (\phi_1(\varepsilon), \ldots, \phi_{r}(\varepsilon))\mid \phi_1(\varepsilon) = g \in S\}.$$
Instead of $(\phi_1(\varepsilon), \ldots, \phi_{r}(\varepsilon))$, we will usually write $(\phi_1(g), \ldots, \phi_{q+r}(g))$ or, since $\phi_1$ is the identity, even $(g,\phi_2(g), \ldots, \phi_{r}(g))$. The isometries $\phi_1(g), \ldots, \phi_{q+r}(g)$ are called \textit{$\phi$-conjugates}. 

Note that $g^*$ and $S^*$ depend on the chosen embeddings $\phi_i$ of $A$ into $M(2,\R)$. On the other hand, the type of $g^*$ is determined uniquely by the type of $g$. This is given by the following lemma.

\begin{lemma}[\cite{sG10a}]
Let $S$ be a subgroup of $\Gamma(A,\mathcal{O})$ and $S^*$ be defined as above. For an element $g\in S$ the following assertions are true.

1. If $g$ is the identity, then $g^*$ is the identity.

2. If $g$ is parabolic, then $g^*$ is parabolic.

3. If $g$ is elliptic of finite order, then $g^*$ is elliptic of the same order.

4. If $g$ is hyperbolic, then $g^*$ is either hyperbolic or mixed such that, for $i = 1,\ldots,r$, $\phi_i(g)$ is either hyperbolic or elliptic of infinite order.

5. If $g$ is elliptic of infinite order, then its $\phi$-conjugates are hyperbolic or elliptic of infinite order.
\end{lemma}

Hence the mixed isometries in this setting have components that are only hyperbolic or elliptic of infinite order. This justifies the condition in our definition of nonelementary that the projections of all mixed isometries can be only hyperbolic or elliptic of infinite order.

By Borel \cite{aB81}, Section 3.3, all \textit{irreducible arithmetic subgroups} of the group $\PSL(2,\R)^r$ are commensurable to a $\Gamma(A,\mathcal{O})^*$. They have finite covolume. By Margulis, for $r\geq2$, all irreducible discrete subgroups of $\PSL(2,\R)^r$ of finite covolume are arithmetic, which shows the importance of the above construction.





\subsection{Properties of nonelementary subgroups of $\PSL(2,\R)^r$}

In this section we cite several results that are used later in the article and whose proofs can be found in \cite{sG10a}.

\begin{lemma}[\cite{sG10a}]
\label{L:Schottky}
Let $\Gamma$ be a nonelementary subgroup of $\PSL(2,\R)^r$. Further let $g$ and $h$ be two hyperbolic isometries in $\Gamma$. Then there are hyperbolic isometries $g'$ and $h'$ in $\Gamma$ with $L(g) = L(g')$ and $L(h) = L(h')$ such that the groups generated by the corresponding components are all Schottky groups (with only hyperbolic isometries).
\end{lemma}

\begin{lemma}[\cite{sG10a}]
\label{L:Nonempty}
Let $\Gamma$ be a subgroup of $\PSL(2,\R)^r$ such that all mixed isometries in $\Gamma$ have only elliptic and hyperbolic components and $p_j(\Gamma)$ is nonelementary for at least one $j\in\{1,\ldots,r\}$. Then $\mathcal{L}_{\Gamma}^{reg}$ is not empty.
\end{lemma}


We define the limit cone of $\Gamma$ to be the closure in $\RP^{r-1}$ of the set of the translation directions of the hyperbolic and mixed isometries in $\Gamma$. This definition is equivalent to the definition of the limit cone given by Benoist in \cite{yB97}. This is explained in Section 3.4 in \cite{sG10a}.

\begin{proposition}[\cite{sG10a}]
\label{P:ConvexAlg}
Let $\Gamma$ be a nonelementary subgroup of an irreducible arithmetic group in $\PSL(2,\R)^r$ with $r \geq 2$. Then $P_\Gamma$ is convex and the closure of $P_\Gamma$ in $\RP^{r-1}$ is equal to the limit cone of $\Gamma$ and in particular the limit cone of $\Gamma$ is convex.
\end{proposition}

\section{The limit set of Zariski dense subgroups}
\label{S:Modular}
The aim of this section is Theorem~\ref{T:FurstenbergLimitSets} which describes the Furstenberg limit set of Zariski dense subgroups $\Gamma$ of irreducible lattices in $\PSL(2,\R)^r$ with $r \geq 2$.

The next two theorems are results in this direction in the special cases when one projection of $\Gamma$ is an arithmetic Fuchsian group or a triangle Fuchsian group.

\begin{theorem}[\cite{sG10a}]
\label{T:MainCharactAlgR}
Let $\Delta$ be an irreducible arithmetic subgroup of $\PSL(2,R)^r$ with $r\geq2$ and $\Gamma$ a finitely generated nonelementary subgroup of $\Delta$. Then $\mathcal{L}_{\Gamma}^{reg}$ is not empty and $P_{\Gamma}$ is embedded homeomorphically in a circle if and only if $p_j(\Gamma)$ is contained in an arithmetic Fuchsian group for some $j\in\{1,\ldots,r\}$. 
\end{theorem}

\begin{theorem}[\cite{sG12}]
Let $\Gamma$ be a Zariski dense subgroup of an irreducible arithmetic group in $\PSL(2,\R)^r$ with $r\geq 2$ such that $p_j(\Gamma)$ is a triangle Fuchsian group.  Then 
\begin{itemize}
\item[(i)] the Furstenberg limit set $F_\Gamma$ is the whole Furstenberg boundary $(\partial \bbH^2)^r$,
\item[(ii)] the limit set $\Lim_\Gamma$ contains an open subset of the geometric boundary of $(\bbH^2)^r$.
\end{itemize}
\end{theorem}

The next lemma is needed in the proof of Theorem~\ref{T:FurstenbergLimitSets}.

\begin{lemma}
\label{L:OneElliptic2}
Let $\Gamma$ be a Zariski dense subgroup of an irreducible arithmetic subgroup of $\PSL(2,\R)^r$ with $r\geq 2$. Further let $g=(g_1,\ldots,g_r)\in\Gamma$ be a mixed isometry with $g_1,\ldots, g_{k-1}$ hyperbolic and $g_k,\ldots, g_r$ elliptic of infinite order with $2 \leq k \leq r-1$. Then there exists a mixed isometry $\tilde{g}\in\Gamma$ such that $\tilde{g}_1,\ldots, \tilde{g}_k$ are hyperbolic and $\tilde{g}_{k+1}$ is elliptic of infinite order.
\end{lemma}
\begin{proof}

Since $\Gamma$ is nonelementary, by Lemma~\ref{L:Nonempty} and by the fact that the fixed points of hyperbolic isometries are dense in $\Lim_\Gamma$, there is $h\in \Gamma$ that is a hyperbolic isometry.
By Lemma~\ref{L:Schottky}, we can assume without loss of generality that for all $i=1,\ldots,k-1$, $h_i$ and $g_i$ generate a Schottky group.

The idea is to find $m\in\N$ such that $g_k^mh_k$ is hyperbolic and $g_{k+1}^mh_{k+1}$ is elliptic (of infinite order).

By \S 7.34 in the book of Beardon~\cite{aB95}, the isometries $h_i$, $i=k,k+1$, can be represented as $h_i=\sigma_{i,2}\sigma_{i,1}$ where $\sigma_{i,j}$ are the reflections in geodesics $L_{i,j}$ that are orthogonal to the axis of $h_i$ and the distance between them equals $\ell(h_i)/2$. By \S 7.33 in \cite{aB95}, the elliptic isometries $g_i^m$ can be represented as $g_i^m=\sigma_{i,4}\sigma_{i,3}$ where $\sigma_{i,j}$ are the reflections in $L_{i,j}$ that pass through the fixed point of $g_i^m$ and the angle between them equals the half of the angle of rotation of $g_i^m$. We can choose $L_{i,2}$ and $L_{i,3}$ to be the same geodesic by choosing them to be the unique geodesic passing through the fixed point of $g_i^m$ that is orthogonal to the axis of $h_i$. Then $\sigma_{i,2}=\sigma_{i,3}$ and
$$ g_i^mh_i = (\sigma_{i,4}\sigma_{i,3})(\sigma_{i,2}\sigma_{i,1})=\sigma_{i,4}\sigma_{i,1}.$$
Hence if $L_{i,1}$ and $L_{i,4}$ do not intersect and do not have a common point at infinity, the isometry $g_i^mh_i$ is hyperbolic, and if they intersect, $g_i^mh_i$ is elliptic. 

\textbf{Aim.} Show that we can choose $m$ so that $g_k^m$ is a rotation of such an angle that $L_{k,4}$ does not intersect $L_{k,1}$ and $g_{k+1}^m$ is a rotation such that $L_{k+1,4}$ intersects $L_{k+1,1}$.

\begin{figure}
\centering

\begin{tikzpicture}[scale=2.3]
\draw (0,0) circle (1) ;

\draw (0,0) node {$\bullet$};
\draw (0,0) node[above left]{Fix($g_k$)};

\draw \geoLine;
\draw (0,-1) node[below]{$L_{k,2}=L_{k,3}$};

\draw [rotate=-40] \geoLine;
\draw (50:0.96) node[above right]{$L_{k,4}$};

\draw [rotate=60] \geoBig;
\draw (-37:1) node[right]{Axis$(h_{k})$};

\draw [rotate=95] \geoSmall;
\draw (-55:0.97) node[below right]{$L_{k,1}$};
\end{tikzpicture}
\begin{tikzpicture}[scale=2.3]
\draw (0,0) circle (1) ;

\draw (0,0) node {$\bullet$};
\draw (0,0) node[above right]{Fix($g_{k+1}$)};

\draw \geoLine;
\draw (0,-1) node[below]{$L_{k+1,2}=L_{k+1,3}$};

\draw [rotate=60] \geoLine;
\draw (-33:1) node[right]{$L_{k+1,4}$};

\draw [rotate=45] \geoMedium;
\draw (-45:0.93) node[below right]{Axis$(h_{k+1})$};

\draw [rotate=82] \geoSmall;
\draw (-10:1) node[right]{$L_{k+1,1}$};
\end{tikzpicture}

\caption{The case with $\tr(g_k) \neq \tr(g_{k+1})$.}
\label{F:DifferentTraces}
\end{figure}
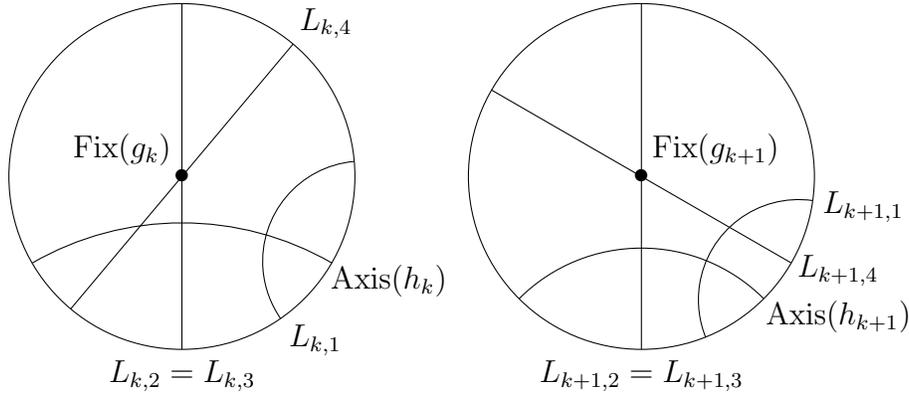

\textbf{Case $\tr(g_k) \neq \tr(g_{k+1})$.} First we note that for all $n \neq 0$ we have $\tr(g_k^n) \neq \tr(g_{k+1}^n)$. Indeed, suppose that it is not the case, i.e. there is a power $n$ such that $\tr(g_k^n) = \tr(g_{k+1}^n)$. The traces $\tr(g_1)$ and $\tr(g_k)$ have the same number $l$ of equal conjugates. Since $g_1$ is hyperbolic, $\tr(g_1^n) = \tr(g_i^n)$ if and only if $\tr(g_1)=\tr(g_i)$ for $i=2,\ldots,r$. Hence $\tr(g_1^n)$ has exactly $l$ equal conjugates. On the other hand $\tr(g_k^n)=\tr(g_i^n)$ if $\tr(g_1)=\tr(g_i)$ for $i=2,\ldots,r$ and additionally $\tr(g_k^n) = \tr(g_{k+1}^n)$. Therefore $\tr(g_k^n)$ has at least $l+1$ equal conjugates which is impossible.

Now we use the fact that every orbit of $(g_k,g_{k+1})$ on the 2-torus $(\partial\bbH^2)^2$ is dense. For a proof see for example Lemma 3.3 in \cite{sG12}. This gives us the possibility to choose $m$ so that $g_k^m$ is a rotation of a very small angle such that $L_{k,4}$ does not intersect $L_{k,1}$ and $g_{k+1}^m$ is a rotation such that $L_{k+1,4}$ intersects $L_{k+1,1}$, see Fig.~\ref{F:DifferentTraces} . 

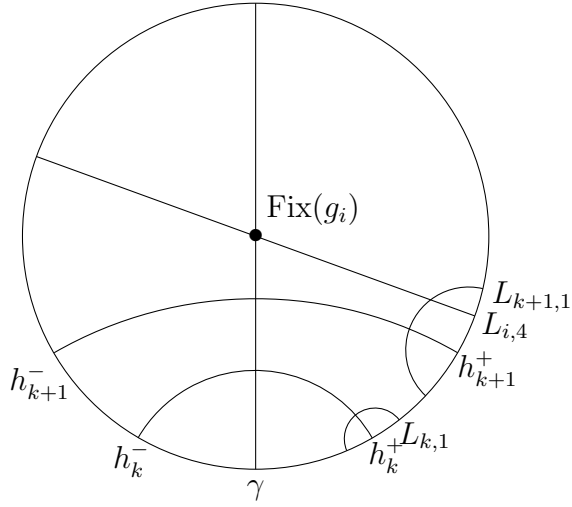
\begin{figure}
\centering
\begin{tikzpicture}[scale=3.1]
\draw (0,0) circle (1) ;

\draw (0,0) node {$\bullet$};
\draw (0,0) node[above right]{Fix($g_{i}$)};

\draw \geoLine;
\draw (0,-1) node[below]{$\gamma$};

\draw [rotate=70] \geoLine;
\draw (-23:1) node[right]{$L_{i,4}$};

\draw [rotate=60] \geoBig;
\draw (-35:1) node[right]{$h_{k+1}^+$};
\draw (-140:0.98) node[left]{$h_{k+1}^-$};

\draw [rotate=30] \geoSmall;
\draw (-62:0.93) node[below right]{$h_{k}^+$};
\draw (-116:0.93) node[below left]{$h_{k}^-$};

\draw [rotate=77] \geoTiny;
\draw (-15:1) node[right]{$L_{k+1,1}$};

\draw [rotate=38] \geoVeryTiny;
\draw (-53:0.94) node[below right]{$L_{k,1}$};
\end{tikzpicture}
\caption{The case with $\tr(g_k) = \tr(g_{k+1}) \text{ and } h_k^+ \neq h_{k+1}^+$.}
\label{F:EqualTracesA}
\end{figure}

\textbf{Case $\tr(g_k) = \tr(g_{k+1})$.} Since $\Gamma$ is Zariski dense, by a theorem of Benoist~\cite{yB97}, $P_{\Gamma}$ is of nonempty interior. Therefore we can choose the hyperbolic element from above so that $\tr(h_k) \neq \tr(h_{k+1})$.

Without loss of generality, after conjugation by an appropriate isometry of $(\bbH)^2$, we can assume the following picture: $g_k$ and $g_{k+1}$ are rotations in the same direction and they both fix the center in the unit disc model of the hyperbolic plane. Furthermore, we assume that the unique geodesic $\gamma$ passing through the center that is orthogonal to the axis of $h_k$ coincides with the corresponding unique geodesic passing through the center that is orthogonal to the axis of $h_{k+1}$. Hence $\gamma=L_{k,2}=L_{k,3} = L_{k+1,2}=L_{k+1,3}$.

First we notice that the attractive and repulsive fixed points of $h_j$, for $j \in \{k, k+1\}$, are symmetric with respect to $\gamma$. Hence  it is not possible that only the attractive or only the repulsive fixed points of $h_k$ and $h_{k+1}$ coincide, i.e. it is not possible to have $h_k^+ = h_{k+1}^+$ and $h_k^- \neq h_{k+1}^-$ or to have  $h_k^+ \neq h_{k+1}^+$ and $h_k^- = h_{k+1}^-$.

If $h_k$ and $h_{k+1}$ have different attractive fixed points, i.e. $h_k^+ \neq h_{k+1}^+$ (and hence $h_k^- \neq h_{k+1}^-$), then we can take $h^n$ instead of $h$ for some power $n$ so that $L_{k,1} \cap L_{k+1,1}=\emptyset$, $L_{k,1} \cap \text{Axis}(h_{k+1})=\emptyset$ and $L_{k+1,1} \cap \text{Axis}(h_{k})=\emptyset$. And since $g_k$ and $g_{k+1}$ are rotations of the same irrational angle, there exists $m$ such that $L_{k,4}$ does not intersect $L_{k,1}$ and $L_{k+1,4}$ intersects $L_{k+1,1}$ (see Fig.~\ref{F:EqualTracesA}).

If $h_k$ and $h_{k+1}$ have the same attractive fixed points, i.e. $h_k^+ = h_{k+1}^+$ (and hence $h_k^- = h_{k+1}^-$), then $L_{k,1}$ and $L_{k+1,1}$ do not intersect because $\tr(h_k) \neq \tr(h_{k+1})$. In this case, since $g_k$ and $g_{k+1}$ are rotations of the same irrational angle, there exists $m$ such that $L_{k,4}$ intersects both $L_{k,1}$ and $L_{k+1,1}$ under different angles (see Fig.~\ref{F:EqualTracesB}). So for the isometry $g'=g^mh$, for $i=1,\ldots,k-1$ the components $g'_i$ are hyperbolic as elements in Schottky groups and $g'_k$ and $g'_{k+1}$ are elliptic with $\tr(g'_k) \neq \tr(g'_{k+1})$. Hence instead of $g$ we can take $g'$ and we use the previous big case.

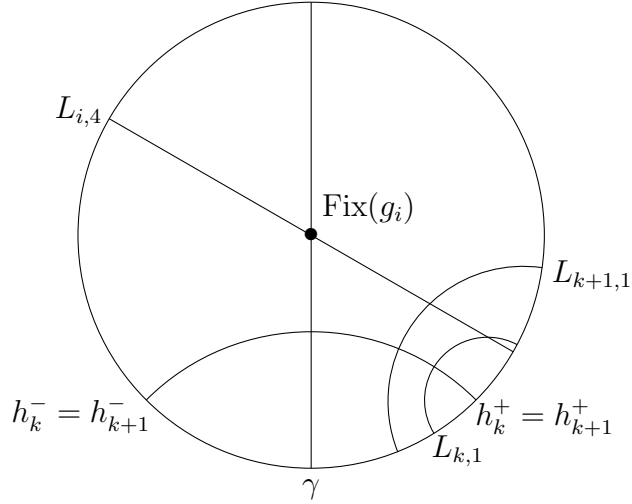
\begin{figure}
\centering
\begin{tikzpicture}[scale=3.1]
\draw (0,0) circle (1) ;

\draw (0,0) node {$\bullet$};
\draw (0,0) node[above right]{Fix($g_{i}$)};

\draw \geoLine;
\draw (0,-1) node[below]{$\gamma$};

\draw [rotate=60] \geoLine;
\draw (148:1) node[left]{$L_{i,4}$};

\draw [rotate=45] \geoMedium;
\draw (-45:0.93) node[below right]{$h_{k}^+ = h_{k+1}^+$};
\draw (-135:0.91) node[below left]{$h_{k}^- = h_{k+1}^-$};

\draw [rotate=82] \geoSmall;
\draw (-10:1) node[right]{$L_{k+1,1}$};

\draw [rotate=62] \geoTiny;
\draw (-60:0.94) node[below right]{$L_{k,1}$};
\end{tikzpicture}
\caption{The case with $\tr(g_k) = \tr(g_{k+1}) \text{ and } h_k^+ = h_{k+1}^+$.}
\label{F:EqualTracesB}
\end{figure}
\end{proof}

Our aim is to describe the possible Furstenberg limit sets of subgroups of irreducible lattices in $\PSL(2,\R)^r$ with $r \geq 2$. For a general Zariski dense group, we obtain the following result.

\begin{theorem}
\label{T:FurstenbergLimitSets}
Let $\Gamma$ be a finitely generated Zariski dense subgroup of an irreducible arithmetic subgroup of $\PSL(2,\R)^r$ with $r \geq 2$.

If there is $j \in \{1,\ldots,r\}$ such that $p_j(\Gamma)$ is discrete, then $F_\Gamma$ is homeomorphic to $\Lim_{p_j(\Gamma)}\times (\partial \bbH^2)^m$, where $m$ is the number of nondiscrete projections of $\Gamma$.

Otherwise, if $p_i(\Gamma)$ is not discrete for all $i=1,\ldots,r$, then $F_\Gamma$ is  the whole Furstenberg boundary.
\end{theorem}
\begin{proof} For simplicity, we denote $\Gamma_i :=p_i(\Gamma)$, for $i=1,\ldots,r$.

First we remark that by Lemma~\ref{L:Nonempty} the regular limit set of $\Gamma$ is not empty and hence there is $\xi = (\xi_1,\ldots,\xi_r) \in F_{\Gamma}$. Without loss of generality, we can assume that $\xi$ is the projection in the Furstenberg boundary of the attractive fixed point of a hyperbolic isometry in $\Gamma$. By Theorem~\ref{T:LinkFP}, $F_{\Gamma}$ is the minimal closed $\Gamma$-invariant subset of the Furstenberg boundary $(\partial \bbH^2)^r$. Therefore $F_{\Gamma} = \overline{\Gamma(\xi)}$. 

We consider first the case where all projections are nondiscrete and then the case where there is $j \in \{1,\ldots,r\}$ such that $p_j(\Gamma)$ is discrete.
 
\subsubsection*{The case without discrete projections} 
The idea is to show by induction for $n=1,\ldots,r$ that $\overline{\Gamma(\xi)}$ is the whole Furstenberg boundary $(\partial \bbH^2)^r$ and hence $F_{\Gamma} = (\partial \bbH^2)^r$.

For $i=1,\ldots,r$, we note by $\Gamma_{1,i}$ the projection of $\Gamma$ to the first $i$ factors.  For simplicity of the notation, we write $F_{G}$ instead of $\Lim_G$ if $G$ is a subgroup of $\PSL(2, \R)$.

For $n=1$ it is clear that $F_{\Gamma_{1,1}} = \partial \bbH^2$ because a nondiscrete nonelementary subgroup of $\PSL(2,\R)$ contains elliptic elements of infinite order. An elliptic element $e$ of infinite order acts on $\partial\bbH^2$ as a "rotation" of irrational angle. That is why the orbit of a point in $\partial\bbH^2$ under the action of $e$ is dense in $\partial\bbH^2$.

Let us assume that $\overline{\Gamma_{1,n-1}((\xi_{1}\ldots,\xi_{n-1}))} = (\partial \bbH^2)^{n-1}$. 

By Lemma~\ref{L:OneElliptic2} we can find $g=(g_{1},\ldots,g_r)\in \Gamma$ such that $g_{1},\ldots,g_{n-1}$ are hyperbolic and $g_n$ is elliptic of infinite order. Let $\eta_i$ denote the attractive fixed point of $g_i$ for $i=1,\ldots,n-1$. We can assume that $\xi_i$ is none of the fixed points of $g_i$. The reason is that $\xi$ is the projection on the Furstenberg boundary of the attractive fixed point of an element $\tilde{g}$ in $\Gamma$. By Lemma~\ref{L:Schottky}, we can find $\tilde{g}'$ such that $\tilde{g}'_i$ and $g_i$ do not have any common fixed point for $i=1,\ldots,n-1$. Then instead of the attractive fixed point of $\tilde{g}$ we can take the attractive fixed point of $\tilde{g}'$. 

We note that $\{\eta_{1}\}\times\ldots\times\{\eta_{n-1}\}\times\partial\bbH^2$ is a subset of $\overline{\Gamma_{1,n}(\xi)}$. The reason is that for any point $\eta_n \in \bbH^2$ there is a sequence $\{n_k\}$ of powers of the elliptic element of infinite order $g_n$ such that $g_n^{n_k}(\xi_n)\longrightarrow \eta_n$ when $n_k\longrightarrow \infty$ and additionally, because of the dynamical properties of the hyperbolic isometries $g_i$ for all $i = 1, \ldots, n-1$, we have $g_i^{n_k}(\xi_i)\longrightarrow \eta_i$ when $n_k\longrightarrow \infty$. 

Hence by the induction hypothesis, for any point $(\zeta_{1},\ldots,\zeta_{n-1})$ in  $(\partial \bbH^2)^{n-1}$, the points in $\{\zeta_{1}\}\times\ldots\times\{\zeta_{n-1}\}\times\partial\bbH^2$ are in the closure of the orbits of $\{\eta_{1}\}\times\ldots\times\{\eta_{n-1}\}\times\partial\bbH^2$. Therefore we have $\overline{\Gamma_{1,n}((\xi_{1}\ldots,\xi_{n}))} = (\partial \bbH^2)^{n}$.

\subsubsection*{The case with discrete projections}
Now we consider the case where there is $j \in \{1,\ldots,r\}$ such that $p_j(\Gamma)$ is discrete. Without loss of generality we can assume that $\Gamma_i$ is discrete for $i=1,\ldots,q$ and nondiscrete for $i=q+1,\ldots,r$ where $q = r - m$. We denote by $\Gamma^{disc}$ the projection of $\Gamma$ to the first $q$ factors (the discrete factors) and by $\Gamma^{ndisc}$ the projection of $\Gamma$ to the rest of the factors.

The plan is to prove first  (Step 1) that  $F_{\Gamma^{disc}}$ is homeomorphic to $\Lim_{\Gamma_1}$. If there are no nondiscrete projections, this finishes the proof. Otherwise, we prove that $F_{\Gamma} = F_{\Gamma^{disc}}  \times F_{\Gamma^{ndisc}}$  (this is Step 2).

\textbf{Step 1.} For all $i=1,\ldots,q$, the natural isomorphism $\phi_i:=p_i \circ p_1^{-1}: \Gamma_1 \rightarrow \Gamma_i$ is type preserving. Since $\Gamma$ and hence $\Gamma_i$ are finitely generated, $\Gamma_i$ is geometrically finite. Hence, by Theorem~3.3 in the paper of Tukia~\cite{pT85}, there is a unique homeomorphism $f_i$ between $\Lim_{\Gamma_1}$ and $\Lim_{\Gamma_i}$ such that $f_i(g_1(x))=\phi_i(g_1)(f_i(x))$ for all $g_1 \in \Gamma_1$ and $x \in \bbH^2 \cup \partial \bbH^2$. We remark that $f_i$ maps the attractive fixed points of hyperbolic elements $g_1$ in $\Gamma_1$ to  the attractive fixed points of hyperbolic elements $g_i=\phi_i(g_1)$ in $\Gamma_1$.

We consider the map $f: \Lim_{\Gamma_1} \rightarrow (\partial \bbH^2)^q$, $\xi_1 \mapsto (\xi_1, f_2(\xi_1), \ldots, f_q(x_q))$. It is a homeomorphism on its image, i.e. $f: \Lim_{\Gamma_1} \rightarrow f( \Lim_{\Gamma_1})$ is a homeomorphism. By the remark above, $f$ is a bijection between the attractive fixed points of the hyperbolic isometries in $\Gamma_1$ and the projections to the Furstenberg boundary of the attractive fixed points of $\Gamma^{disc}$. Since the attractive fixed points of hyperbolic elements are dense in the corresponding limit sets, we have $F_{\Gamma^{disc}}=f(\Lim_{\Gamma_1})$. Therefore $f: \Lim_{\Gamma_1} \rightarrow  F_{\Gamma^{disc}}$ is a homeomorphism.

\textbf{Step 2.} The idea is to show by induction on the nondiscrete factors that $\overline{\Gamma(\xi)}$ and hence $F_{\Gamma}$ is  $F_{\Gamma^{disc}}  \times (\partial \bbH^2)^m$. The induction step is essentially the same as in the case without discrete projections. The difference is in the induction base (for which we consider only the the first nondiscrete factor).

By Step 1, there is a bijection between $F_{\Gamma^{disc}}$ and $\Lim_{\Gamma_1}$ compatible with the $\Gamma$-action. So, for convenience of the notation, we can consider only the case with one discrete factor.

Let $g = (g_1,g_2)\in \Gamma$ be a transformation such that $g_2$ is an elliptic transformation of infinite order. Such a $g_2$ exists because $\Gamma_2$ is not discrete. The isometry $g_1$ is hyperbolic because $\Gamma_1$ is discrete. 

Let $\eta_1$ be the attractive fixed point of $g_1$. First we note that $\{\eta_1\}\times\partial\bbH^2$ is in $\overline{\Gamma(\xi)}$. The reason is that for any point $\eta_2 \in \bbH^2$ there is a sequence $\{n_k\}$ of powers such that $g_2^{n_k}(\xi_2)\longrightarrow \eta_2$ when $n_k\longrightarrow \infty$ and additionally, because of the dynamical properties of the hyperbolic isometries, $g_1^{n_k}(\xi_1)\longrightarrow \eta_1$ when $n_k\longrightarrow \infty$. The only problem could be that $\xi_1$ is the repelling fixed point of $g_1$. In this case we consider $g^{-1}$ instead of $g$. 

Now let $\zeta = (\zeta_1,\zeta_2)$ be a point in $\Lim_{\Gamma_1}\times\partial\bbH^2$. Since $\Lim_{\Gamma_1}=\overline{\Gamma_1(\eta_1)}$, there is a sequence of elements $\{h_n\}$ in $\Gamma$ such that $(h_n)_1(\eta_1)\stackrel{n\rightarrow \infty}\longrightarrow \zeta_1$.

The points $(\eta_1,(h_n)_2^{-1}(\zeta_2))$ are points in $\overline{\Gamma(\xi)}$ because it contains $\{\eta_1\}\times\partial\bbH^2$. Hence all the points $((h_n)_1(\eta_1), \zeta_2)$ are in $\overline{\Gamma(\xi)}$. Therefore their limit $\zeta=\lim_{n\rightarrow \infty}((h_n)_1(\eta_1), \zeta_2)$ is also in $\overline{\Gamma(\xi)}$. Thus $\overline{\Gamma(\xi)}=\Lim_{\Gamma_1}\times\partial\bbH^2$.
\end{proof}

\begin{rem}
This theorem together with the fact that the projective limit set of a Zariski dense subgroup of $\PSL(2,\R)^r$ is of nonempty interior (this is a special case of the theorem by Benoist in Section 1.2 in \cite{yB97}) describes the regular limit set of finitely generated Zariski dense subgroups of $\PSL(2,\R)^r$.
\end{rem}
\begin{rem}
By Theorem~4.10 of Link \cite{gL06}, the attractive fixed points of the hyperbolic isometries in a nonelementary subgroup of $\PSL(2,\R)^r$ are dense in the limit set. Therefore the regular limit set of a nonelementary subgroup of $\PSL(2,\R)^r$ is dense in the limit set. Thus by the above theorem and the previous remark we have also a description of the whole limit set of Zariski dense subgroups of $\PSL(2,\R)^r$.
\end{rem}

\section{The limit set of finitely generated and nonelementary subgroups}
The aim of this section is to describe the limit set in the general case, i.e. the limit set of finitely generated and nonelementary subgroups of irreducible lattices in $\PSL(2,\R)^r$ with $r \geq 2$. We do it with Proposition~\ref{P:FurstenbergLimitSets} and Proposition~\ref{P:ProjectiveLimitSets}.

First we recall a result about the Zariski closure of subgroups of $\PSL(2,\R)^r$.

\begin{proposition}[\cite{sG12}]
\label{P:ZariskiClosure}
Let $\Gamma$ be a nonelementary subgroup of $\PSL(2,\R)^r$. Then the following holds for the Zariski closure of $\Gamma$.
$$ \overline{\Gamma}^Z=\prod_{i=1}^n \Diag_{k_i}(\PSL(2,\R)),$$
where $\Diag_{k_i}(\PSL(2,\R))$ is a conjugate of the diagonal embedding of $\PSL(2,\R)$ in $\PSL(2,\R)^{k_i}$ and where $k_1+\cdots+k_n=r$.
\end{proposition}

For a group $S$ we denote as usual by $S^{(2)}$ the subgroup of $S$ generated by the set $\{g^2 \mid g \in S\}$. If $S$ is a finitely generated nonelementary subgroup of $\PSL(2,\R)$ then $S^{(2)}$ is a finite index normal subgroup of $S$.

\begin{corollary} [\cite{sG12}]
Let $\Gamma$ be a finitely generated nonelementary subgroup of an irreducible arithmetic group $\Delta$ in $\PSL(2,\R)^r$. Then $\Gamma$ is Zariski dense in $\PSL(2,\R)^r$ if and only if the fields generated by $\Tr(p_i(\Gamma^{(2)}))$ and $\Tr(p_i(\Delta^{(2)}))$ are equal for one and hence for all $i \in \{1,\ldots,r\}$.
\end{corollary}

We can generalize this corollary in the following way.
\begin{corollary}
\label{C:GenExtZariski}
Let $\Gamma$ be a finitely generated nonelementary subgroup of an irreducible arithmetic group $\Delta$ in $\PSL(2,\R)^r$. Then for the Zariski closure of $\Gamma$ we have
$$ \overline{\Gamma}^Z = \prod_{i=1}^n \Diag_{k}(\PSL(2,\R)),$$ 
where $k = [\Q(\Tr(p_j(\Delta^{(2)})) : \Q(\Tr(p_j(\Gamma^{(2)})))],$
for one and hence for all $j \in \{1,\ldots,r\}$.
\end{corollary}

\begin{proof}
By the proof of the previous corollary in \cite{sG12}, we can assume without loss of generality that $\Delta$ is an arithmetic group derived from a quaternion algebra $\Gamma(A,\mathcal{O})^*$. In this case $\Gamma$ is equal to $S^*$, where $S$ is a subgroup of $\Gamma(A,\mathcal{O})^*$.

We denote by $K$ the field generated by $\Tr(\Gamma(A,\mathcal{O})^{(2)})$ and by $F$ the field generated by $\Tr(S^{(2)})$.

First we remark that, by Lemma~\ref{L:Nonempty}, $\phi_i(\Gamma(A,\mathcal{O})^{(2)})$ contains hyperbolic elements if and only if $\phi_i(S^{(2)})$ contains hyperbolic elements, for all $i=1,\ldots,m$, where $m$ is the degree of $K$ over $\Q$. This means that for each Galois embedding $\sigma$ of $F$ into $\R$, that is the restriction of some $\phi_i$, $i=1,\ldots,r$, there are exactly $k=[K:F]$ embeddings $\phi_l$ with $l \in \{1, \ldots, r\}$, such that $\sigma$ is the restriction of $\phi_l$ to $F$.

The statement of the corollary follows from the above considerations and the fact that if the restrictions to $F$ of $\phi_{k_1}$ and $\phi_{k_2}$ are the same, then $\phi_{k_1}(S)=\phi_{k_2}(S)$.
\end{proof}

Notice that in this corollary $n$ is equal to $ [\Q(\Tr(p_j(\Delta^{(2)})) : \Q]$ if and only if $r$ is equal to $[\Q(\Tr(p_j(\Gamma^{(2)}))) : \Q]$.

Using the theorem in the previous section (Theorem~\ref{T:FurstenbergLimitSets}) and Corollary~\ref{C:GenExtZariski}, we prove the following proposition, which describes the Furstenberg limit set of a finitely generated and nonelementary subgroup of an irreducible arithmetic group in $\PSL(2,\R)^r$ with $r \geq 2$.

\begin{proposition}
\label{P:FurstenbergLimitSets}
Let $\Gamma$ be a finitely generated and nonelementary subgroup of an irreducible arithmetic subgroup $\Delta$ of $\PSL(2,\R)^r$ with $r \geq 2$. Further let $m$ be the number of nondiscrete projections of $\Gamma$ and $k$ the degree of extension $k = [\Q(\Tr(p_j(\Delta^{(2)})) : \Q(\Tr(p_j(\Gamma^{(2)})))]$
for one and hence for all $j \in \{1,\ldots,r\}$.

If there is $j \in \{1,\ldots,r\}$ such that $p_j(\Gamma)$ is discrete, then $F_\Gamma$ is homeomorphic to $\Lim_{p_j(\Gamma)}\times (\partial \bbH^2)^\frac{m}{k}$, otherwise,  $F_\Gamma$ is homeomorphic to $(\partial \bbH^2)^\frac{m}{k}$.
\end{proposition}

The description of the projective limit set comes from the following proposition, which is an immediate corollary of Corollary~\ref{C:GenExtZariski} and the theorem by Benoist in Section 1.2 in \cite{yB97} that states in our particular case that the projective limit set of a Zariski dense subgroup of $\PSL(2,\R)^r$ is of nonempty interior.
\begin{proposition}
\label{P:ProjectiveLimitSets}
Let $\Gamma$ be a finitely generated nonelementary subgroup of an irreducible arithmetic subgroup $\Delta$ of $\PSL(2,\R)^r$ with $r \geq 2$. Then $P_{\Gamma}$ is of dimension $\frac{r}{k}-1$ with $k$ as above.
\end{proposition}

The above two propositions give us a relation between the dimensions of $F_\Gamma$ and $P_\Gamma$:
\begin{corollary}
\label{C:IneqDimFP}
Let $\Gamma$ be a finitely generated nonelementary subgroup of an irreducible arithmetic subgroup $\Delta$ of $\PSL(2,\R)^r$ with $r \geq 2$. Then 
$$\dim{F_{\Gamma}} \leq 1+\dim{P_\Gamma} .$$
\end{corollary}

In particular, this corollary implies that it is not possible to have simultaneously a big $F_{\Gamma}$ and a small $P_\Gamma$. For example, it is shown in  \cite{sG10a} that if $P_\Gamma$ is just one point, then there is a projection of $\Gamma$ that is contained in an arithmetic Fuchsian group and hence $F_{\Gamma}$ is the smallest possible, i.e. $F_{\Gamma}$ is embedded homeomorphically in a circle.

Some examples of small $F_{\Gamma}$ and big $P_\Gamma$ can be obtained by the nonarithmetic examples of semi-arithmetic groups constructed by Schmutz and Wolfart in Theorem~1 and Theorem~2 in \cite{pS00}: A semi-arithmetic group is a Fuchsian group that can be embedded naturally in an irreducible lattice in $\PSL(2,\R)^r$. Let $S$ be a semi-arithmetic group as above. From Theorem~1 and Theorem~2 in \cite{pS00} we can deduce that $S$ can be embedded in an irreducible lattice in $\PSL(2,\R)^2$.  We note $\Gamma$ this natural embedding. On one hand, by \cite{pS00}, all projections of $\Gamma$ are discrete and hence $F_{\Gamma}$ is homeomorphic to a circle. On the other hand, $\Gamma$ is Zariski dense and therefore $P_{\Gamma}$ is of nonempty interior. In these examples we have $\dim{F_{\Gamma}}=1$ and $\dim{P_\Gamma}=1$ and so the inequality of Corollary~\ref{C:IneqDimFP} is strict.

Finally we note that it is possible to have infinite covolume groups $\Gamma$ with big limit sets, i.e. big $F_{\Gamma}$ and big $P_\Gamma$. In \cite{sG12} it is shown that the minimal embeddings of groups admitting modular embeddings (e.g. triangle Fuchsian groupes as shown by Cohen and Wolfart in \cite{pC90}) are examples of infinite covolume subgroupes with big $F_{\Gamma}$ and big $P_\Gamma$, i.e. $F_{\Gamma}$ is the whole Furstenberg boundary and $P_{\Gamma}$ is of nonempty interior. In this case we have $\dim{F_{\Gamma}} = 1+\dim{P_\Gamma}$.

\end{document}